\documentclass[12pt]{article}
\usepackage{graphicx}
\usepackage{amssymb, amsthm, amsmath, amsfonts}
\usepackage{mathrsfs,microtype}
\usepackage{epstopdf, color}
\usepackage{tikz}
\usepackage{hyperref}

\usetikzlibrary{arrows,decorations.pathmorphing,decorations.footprints,fadings,calc,trees,mindmap,shadows,decorations.text,patterns,positioning,shapes,matrix,fit,through,decorations.pathreplacing}
\usepackage[margin=1in]{geometry}

\theoremstyle{plain}
\newtheorem{theorem}{Theorem}[section]
\newtheorem{lemma}[theorem]{Lemma}
\newtheorem{corollary}[theorem]{Corollary}

\newtheorem{question}[theorem]{Question}

\theoremstyle{definition}
\newtheorem{definition}[theorem]{Definition}

\theoremstyle{remark}
\newtheorem*{remark}{Remark}

\newcommand{\calT}{\mathcal{T}}
\newcommand{\calC}{\mathcal{C}}

\title{Extremal $H$-colorings of trees and $2$-connected graphs}
\author{John Engbers\thanks{john.engbers@marquette.edu; Department of Mathematics, Statistics and Computer Science, Marquette University, Milwaukee, WI 53201. Research supported by the Simons Foundation and by a Marquette University Summer Faculty Fellowship} \and David Galvin\thanks{dgalvin1@nd.edu; Department of Mathematics,
University of Notre Dame, Notre Dame IN 46556. Research supported by NSA grant H98230-13-1-0248, and by the Simons Foundation.}}
\date{\today }

\begin{document}

\maketitle

\begin{abstract}
For graphs $G$ and $H$, an {\em $H$-coloring} of $G$ is an adjacency preserving map from the vertices of $G$ to the vertices of $H$.  $H$-colorings generalize such notions as independent sets and proper colorings in graphs. There has been much recent research on the extremal question of finding the graph(s) among a fixed family that maximize or minimize the number of $H$-colorings.
In this paper, we prove several results in this area.

First, we find a class of graphs ${\mathcal H}$ with the property that for each $H \in {\mathcal H}$, the $n$-vertex tree that minimizes the number of $H$-colorings is the path $P_n$.
We then present a new proof of a theorem of Sidorenko, valid for large $n$, that for {\em every} $H$ the star $K_{1,n-1}$ is the $n$-vertex tree that maximizes the number of $H$-colorings.
Our proof uses a stability technique which we also use to show that for any non-regular $H$ (and certain regular $H$) the complete bipartite graph $K_{2,n-2}$ maximizes the number of $H$-colorings
of $n$-vertex $2$-connected graphs.
Finally, we show that the cycle $C_n$ has the most proper $q$-colorings among all $n$-vertex $2$-connected graphs.
\end{abstract}

\section{Introduction and statement of results}

For a simple loopless graph $G = (V(G),E(G))$ and a graph $H = (V(H),E(H))$ (possibly with loops, but without multi-edges), an {\em $H$-coloring} of $G$ is an adjacency-preserving map $f:V(G) \to V(H)$ (that is, a map satisfying $f(x) \sim_H f(y)$ whenever $x\sim_G y$). Denote by $\hom(G,H)$ the number of $H$-colorings of $G$.

The notion of $H$-coloring has been the focus of extensive research in recent years. Lov\'asz's monograph \cite{Lovasz} explores natural connections to graph limits, quasi-randomness and property testing. Many important graph notions can be encoded via homomorphisms --- for example, {\em proper $q$-coloring} using $H=K_q$ (the complete graph on $q$ vertices), and {\em independent} (or {\em stable}) {\em sets} using $H=H_{\text{ind}}$ (an edge with one looped endvertex). The language of $H$-coloring is also ideally suited for the mathematical study of hard-constraint spin models from statistical physics (see e.g. \cite{BrightwellWinkler}). Particularly relevant to the present paper is the Widom-Rowlinson model of the occupation of space by $k$ mutually repulsive particles, which is encoded as an $H$-coloring by using the graph $H = H_{\rm WR}(k)$ which has loops on every vertex of $K_{1,k}$ (the star on $k+1$ vertices). Note that the original Widom-Rowlinson model \cite{WidomRowlinson} has $k=2$.

Many authors have addressed the following extremal enumerative question for $H$-coloring: given a family ${\mathcal G}$ of graphs, and a graph $H$, which $G \in {\mathcal G}$ maximizes or minimizes ${\rm hom}(G,H)$? This question can be traced back to Birkhoff's attacks on the 4-color theorem, but recent attention on it owes more to Wilf and (independently) Linial's mid-1980's query as to which $n$-vertex, $m$-edge graph admits the most proper $q$-colorings (i.e. has the most $K_q$-colorings). For a survey of the wide variety of results and conjectures on the extremal enumerative $H$-coloring question, see \cite{Cutler}.

A focus of the present paper is the extremal enumerative $H$-coloring question for the family $\calT(n)$, the set of all trees on $n$ vertices. This family has two natural candidates for extremality, namely the path $P_n$ and the star $K_{1,n-1}$, and indeed in \cite{ProdingerTichy} Prodinger and Tichy showed that these two are
extremal for the count of independent sets in trees: for all $T \in \calT(n)$,
\begin{equation} \label{inq-PT}
\hom(P_n,H_{\text{ind}}) \leq \hom(T,H_{\text{ind}}) \leq \hom(K_{1,n-1},H_{\text{ind}}).
\end{equation}
The Hoffman-London matrix inequality (see e.g. \cite{Hoffman,London,Sidorenko2}) is equivalent to the statement that $\hom(P_n,H) \leq \hom(K_{1,n-1},H)$ for \emph{all} $H$ and $n$; a significant generalization of this due to Sidorenko \cite{Sidorenko} (see also \cite{CsikvariLin} for a short proof) shows that the right-hand inequality of (\ref{inq-PT}) extends to arbitrary $H$.
\begin{theorem}[Sidorenko] \label{thm-siderenko}
Fix $H$ and $n \geq 1$.  Then for any $T \in \calT(n)$,
\[
\hom(T,H) \leq \hom(K_{1,n-1},H).
\]
\end{theorem}
In other words, the star admits not just the most independent sets among $n$-vertex trees, but also the most $H$-colorings for arbitrary $H$. Two points are worth noting here. First, since deleting edges in a graph cannot decrease the number of $H$-colorings, Theorem \ref{thm-siderenko} shows that among all connected graphs on $n$ vertices, $K_{1,n-1}$ admits the most $H$-colorings. Second, if we extend $\calT(n)$ instead to the family of graphs on $n$ vertices with minimum degree at least $1$ and consider even $n$, then as shown by the first author \cite{Engbers} the number of $H$-colorings is maximized either by the star or by the graph consisting of a union of disjoint edges.

The left-hand side of (\ref{inq-PT}) turns out {\em not} to generalize to arbitrary $H$: Csikv\'{a}ri and Lin \cite{CsikvariLin}, following earlier work of Leontovich \cite{Leontovich}, exhibit a (large) tree $H$ and a tree $E_7$ on seven vertices such that $\hom(E_7,H) < \hom(P_7,H)$, and raise the natural question of characterizing those $H$ for which $\hom(P_n,H) \leq \hom(T,H)$ holds for all $n$ and all $T \in \calT(n)$.

Our first result gives a partial answer to this question. Before stating it, we need to establish a convention concerning degrees of vertices in graphs with loops.

\medskip

\noindent \textbf{Convention:} For all graphs in this paper, the degree of a vertex $v$ is the number of neighbors
of
$v$, i.e., $d(v) = |\{w : v \sim w\}|$. In particular, a loop on a vertex
adds one to the degree.  We let $\Delta$ denote the maximum degree of $H$.

We also let $G^{\circ}$ denote the graph obtained from $G$ by adding loops to every vertex in $G$.

\begin{theorem} \label{thm-minfortrees}
Let $n\geq 1$ and let $H$ be a regular graph.  For an integer $\ell \geq 1$, let $H^{\circ}(\ell)$ be the join of $H$ and $K_{\ell}^{\circ}$. 
Then for any $T \in \calT(n)$,
\[
\hom(P_n,H^{\circ}(\ell)) \leq \hom(T,H^{\circ}(\ell)).
\]
Equality occurs if and only if $T = P_n$ or $H^{\circ}(\ell) = K_{q}^{\circ}$ for some $q \geq \ell$. 
\end{theorem}
Notice that the result also holds for $H$ where each component is of the form $H^{\circ}(\ell)$. Theorem \ref{thm-minfortrees} generalizes the left-hand side of (\ref{inq-PT}), as $H_{\text{ind}}$ is the join of $K_1$ and $K_{1}^{\circ}$, 
and our proof is a generalization of the inductive approach used in \cite{ProdingerTichy}.

Since the Widom-Rowlinson graph $H_{\rm WR}(k)$ can be constructed from the disjoint union of $k$ looped vertices by the addition of a single looped dominating vertex, an immediate corollary of Theorem \ref{thm-minfortrees} is that for all $n \geq 1$,
\[
\hom(P_n,H_{\rm WR}(k)) \leq \hom(T,H_{\rm WR}(k)).
\]
for all $T \in \calT(n)$. We also note in passing that Theorem \ref{thm-minfortrees} may be interpreted in terms of partial $H$-colorings of $G$ (that is, $H$-colorings of induced subgraphs of $G$ that need not be extendable to $H$-colorings of $G$). Specifically, when $\ell=1$ the theorem says that if $H$ is regular, then among all $n$-vertex trees none admits fewer partial $H$-colorings than $P_n$.

Our second result is a new proof of Sidorenko's theorem, valid for sufficiently large $n$.
\begin{theorem} \label{thm-maxfortrees}
There is a constant $c_H$ such that if $n \geq c_H$ and $T \in \calT(n)$ then
\[
\hom(T,H) \leq \hom(K_{1,n-1},H)
\]
with equality if and only if $H$ is regular or $T = K_{1,n-1}$.
\end{theorem}

While Theorem \ref{thm-maxfortrees} is weaker than Theorem \ref{thm-siderenko} in that it only holds for $n \geq c_H$, it is noteworthy for two reasons.  Firstly, our proof for non-regular $H$ uses a stability technique --- we show that if a tree is not structurally close to a star (specifically, if it has a long path), then it admits significantly fewer $H$-colorings than the star, and if it is structurally almost a star but has some blemishes, then again it admits fewer $H$-colorings. Secondly, the proof is less tree-dependent than Sidorenko's, and so the ideas used may be applicable in other settings. We illustrate this by considering the extremal enumerative question for $H$-colorings of $2$-connected graphs. Let $\calC_2(n)$ denote the set of $2$-connected graphs on $n$ vertices, and let $K_{a,b}$ be the complete bipartite graph with $a$ and $b$ vertices in the two color sets. Also, for graph $H$ with maximum degree $\Delta$, denote by $s(H)$ the number of ordered pairs $(i,j)$ of vertices of $H$ satisfying $|N(i) \cap N(j)| = \Delta$. The special case $H=H_{\text{ind}}$ of the following was established by Hua and Zhang \cite{HuaZhang}.
\begin{theorem} \label{thm-2connected}
For non-regular connected $H$ there is a constant $c_H$ such that if $n \geq c_H$ and $G \in \calC_2(n)$ then
\[
\hom(G,H) \leq \hom(K_{2,n-2},H)
\]
with equality if and only if $G=K_{2,n-2}$.

For $\Delta$-regular $H$ the same conclusion holds whenever $s(H) \geq 2\Delta^2 + 1$ (when $H$ is loopless and bipartite) or $s(H) \geq \Delta^2+1$ (otherwise).
\end{theorem}
For regular $H$ we have $s(H) \geq |V(H)|$ (consider ordered pairs of the form $(i,i)$), so an immediate corollary in this case is that if $|V(H)| \geq 2\Delta^2+1$ (when $H$ is loopless and bipartite) or if $|V(H)| \geq \Delta^2+1$ (otherwise) then for all large $n$ the unique $2$-connected graph with the most $H$-colorings is $K_{2,n-2}$. Note that the bounds on $s(H)$ in the $\Delta$-regular case are tight with respect to the characterization of uniqueness: when $H=K_{\Delta}^{\circ}$ we have $\hom(G,H) = \Delta^n$ for {\em all} $n$-vertex $G$, and here $s(H)=\Delta^2$, and when $H=K_{\Delta, \Delta}$ (a loopless bipartite graph) we have $\hom(G,H) = 2\Delta^n \cdot \textbf{1}_{\{\text{$G$ bipartite}\}}$ for all $n$-vertex connected graphs $G$, and here $s(H)=2\Delta^2$.

Since all $G \in \calC_2(n)$ are connected, the restriction to connected $H$ is natural.  Unlike with Theorem \ref{thm-maxfortrees}, the restriction to non-regular $H$ is somewhat significant here. In particular, there are examples of regular graphs $H$ such that for all large $n$, there are graphs in $\calC_2(n)$ that admit more $H$-colorings than $K_{2,n-2}$ (and so there is no direct analog of Theorem \ref{thm-siderenko} in the world of $2$-connected graphs). One such example is $H=K_3$; it is easily checked that ${\rm hom}(C_n,K_3)>{\rm hom}(K_{2,n-2},K_3)$ for all $n \geq 6$ (where $C_n$ is the cycle on $n$ vertices).  
More generally, we have the following.
\begin{theorem} \label{thm-2conn-cycles}
Let $n \geq 3$ and $q \geq 3$ be given.  Then for any $G \in \calC_2(n)$, 
\[
\hom(G,K_q) \leq \hom(C_n,K_q),
\]
with equality if and only if $G=C_n$, unless $q=3$ and $n=5$ in which case $G=K_{2,3}$ also achieves equality.

The same conclusion holds with $\calC_2(n)$ replaced by the larger family of $2$-edge-connected graphs on $n$ vertices.
\end{theorem}

In the setting of Theorem \ref{thm-2conn-cycles} it turns out that no extra complications are introduced in moving from $2$-connected to $2$-edge-connected. We do not, however, currently see a way to extend Theorem \ref{thm-2connected} to this larger family. In particular, the key structural lemma given in Corollary \ref{cor-P3} does not generalize to the family of 2-edge-connected graphs. 

The rest of the paper is laid out as follows. The proof of Theorem \ref{thm-minfortrees} is given in Section \ref{sec-minfortrees}. We prove Theorem \ref{thm-maxfortrees} in Section \ref{sec-maxfortrees}, and use similar ideas to then prove Theorem \ref{thm-2connected} in Section \ref{sec-2connected}. The proof of Theorem \ref{thm-2conn-cycles} is given in Section \ref{sec-2conn-cycles}. Finally, we end with a number of open questions in Section \ref{sec-concluding remarks}.

\section{Proof of Theorem \ref{thm-minfortrees}} \label{sec-minfortrees}

Let $H$ be a $\Delta$-regular graph, and let $H^{\circ}(\ell)$ be the join of $H$ and $K_{\ell}^{\circ}$. 
Using induction on $n$, we will show something a little stronger than Theorem \ref{thm-minfortrees}, namely that for any $n$-vertex \emph{forest} $F$, $\hom(F,H^{\circ}(\ell)) \geq \hom(P_n,H^{\circ}(\ell))$, with equality if and only if either $H^{\circ}(\ell)$ is a complete looped graph, or $F=P_n$. We will first prove the inequality, and then address the cases of equality.

For $n \leq 4$, the only forest $F$ on $n$ vertices that is not a subgraph of $P_n$ (and so for which $\hom(F,H^{\circ}(\ell))\geq \hom(P_n,H^{\circ}(\ell))$ is immediate) is $F=K_{1,n-1}$; but in this case the required inequality follows directly from the Hoffman-London inequality \cite{Hoffman,London}.

So now fix $n \geq 5$, and let $F$ be a forest on $n$ vertices. In what follows, for $x \in V(F)$ and $i \in V(H^{\circ}(\ell))$ we write $\hom(F,H^{\circ}(\ell) | x \mapsto i)$ for the number of homomorphisms from $F$ to $H^{\circ}(\ell)$ that map $x$ to $i$.
Also, we'll assume that $H^{\circ}(\ell)$ has $q+\ell$ vertices, say $V(H^{\circ}(\ell)) = \{1,\ldots,q, \ldots,q+\ell\}$, with the looped dominating vertices being $q+1,\ldots,q+\ell$.

Let $x$ be a leaf in $F$, with unique neighbor $y$ (note that we may assume that $F$ has an edge, because the desired inequality is trivial otherwise).
For each $i  \in \{1,\ldots,q+\ell\}$,
\begin{equation} \label{eqn-restrict}
 \hom(F,H^{\circ}(\ell) | x \mapsto i) = \sum_{j \sim i} \hom(F-x,H^{\circ}(\ell) | y \mapsto j),
\end{equation}
which in particular implies $\hom(F,H^{\circ}(\ell) | x \mapsto k) = \hom(F-x,H^{\circ}(\ell))$ if $k \in \{q+1,\ldots,q+\ell\}$.
So
\begin{eqnarray*}
\hom(F,H^{\circ}(\ell)) &=& \sum_{i=1}^{q+\ell} \hom(F,H^{\circ}(\ell) | x \mapsto i) \\
&=& \ell \hom(F-x,H^{\circ}(\ell)) + \sum_{i=1}^{q} \sum_{j\sim i} \hom(F-x,H^{\circ}(\ell) | y \mapsto j)\\
&=& \ell \hom(F-x,H^{\circ}(\ell)) + \sum_{j=q+1}^{q+\ell} q\hom(F-x,H^{\circ}(\ell) | y \mapsto j) \\
&& \qquad + \sum_{j=1}^{q} d_{H}(j) \hom(F-x,H^{\circ}(\ell) | y \mapsto j)\\
&=& \ell \hom(F-x,H^{\circ}(\ell)) + \sum_{j=q+1}^{q+\ell} (q-\Delta) \hom(F-x,H^{\circ}(\ell) | y \mapsto j) \\
&& \qquad + \Delta\sum_{j=1}^{q+\ell} \hom(F-x,H^{\circ}(\ell) | y \mapsto j)\\
&=& (q-\Delta)\hom(F-x-y,H^{\circ}(\ell)) + (\Delta+\ell) \hom(F-x,H^{\circ}(\ell)).
\end{eqnarray*}

If $F=P_n$ then $F-x=P_{n-1}$ and $F-x-y=P_{n-2}$.  Therefore, by induction, we have
\begin{eqnarray*}
\hom(F,H^{\circ}(\ell)) &=& (q-\Delta)\hom(F-x-y,H^{\circ}(\ell)) + (\Delta+\ell)\hom(F-x,H^{\circ}(\ell))\\
&\geq& (q-\Delta)\hom(P_{n-2},H^{\circ}(\ell)) + (\Delta+\ell)\hom(P_{n-1},H^{\circ}(\ell))\\
&=& \hom(P_{n},H^{\circ}(\ell)).
\end{eqnarray*}

When can we achieve equality?  If $H^{\circ}(\ell)$ is the complete looped graph, then we have equality for all $F$. So we may assume that $H^{\circ}(\ell)$ is not a complete looped graph, and that in particular there are $i, j \in V(H^{\circ}(\ell))$ with $i \not \sim j$.
Now consider an $F$ with more than one component, and let $u$ and $v$ be vertices of $F$ in different components. Using the looped dominating vertices of $H^{\circ}(\ell)$ we may construct an $H^{\circ}(\ell)$-coloring of $F$ in which $u$ is colored $i$ and $v$ is colored $j$. This is not a valid $H^{\circ}(\ell)$-coloring of the forest obtained from $F$ by adding the edge $uv$. It follows that there is a tree $T$ on $n$ vertices with $\hom(F,H^{\circ}(\ell)) > \hom(T,H^{\circ}(\ell))$.
The proof of the Hoffman-London inequality ($\hom(K_{1,n-1},H) \geq \hom(P_n,H)$) given in \cite{Sidorenko2} in fact shows strict inequality for $H^{\circ}(\ell)$ that is not a complete looped graph; since strict inequality holds for the base cases $n \leq 4$, the inductive proof therefore gives strict inequality unless $F-x=P_{n-1}$ and $F-x-y=P_{n-2}$, which implies $F=P_n$.


\section{Proof of Theorem \ref{thm-maxfortrees}} \label{sec-maxfortrees}
First, notice that for $\Delta$-regular $H$ we have ${\rm hom}(T,H) = |V(H)| \Delta^{n-1}$ for all $T \in \calT(n)$, as can be seen by fixing the color on one vertex and iteratively coloring away from that vertex. 
Therefore the goal for the remainder of this section is to show that for non-regular $H$ there is a constant $c_H$ such that if $n \geq c_H$ and $T \in \calT(n)$ then
\[
\hom(T,H) \leq \hom(K_{1,n-1},H),
\]
with equality if and only if $T = K_{1,n-1}$.
Recall that a loop in $H$ will count \emph{once} toward the degree of a vertex $v \in V(H)$ and $\Delta$ denotes the maximum degree of $H$.  If $H$ has components $H_1,\ldots,H_s$ then $\hom(G,H)= \hom(G,H_1) + \cdots + \hom(G,H_s)$, so we may assume that $H$ is connected.

Let $V_{=\Delta} \subset V(H)$ be the set of vertices in $H$ with degree $\Delta$.  By coloring the center of the star with a color from $V_{=\Delta}$, we have
\[
\hom(K_{1,n-1},H) \geq |V_{=\Delta}|\Delta^{n-1}.
\]
We will show that if $n \geq c_H$ and $T \neq K_{1,n-1}$, then $\hom(T,H) < |V_{=\Delta}|\Delta^{n-1}$. We use the following lemma, which will also be needed in the proof of Theorem \ref{thm-2connected}.

\begin{lemma}\label{lem-no long paths}
For non-regular $H$ there exists a constant $\ell_H$ such that if $k \geq \ell_H$, then $\hom(P_{k},H) < \Delta^{k-2}$.
\end{lemma}

\begin{proof}
Let $A$ denote the adjacency matrix of $H$, which is a symmetric non-negative matrix.  
Since an $H$-coloring of $P_k$ is exactly a walk of length $k$ through $H$, we have
\begin{equation} \label{eq-homs-lin-alg-count}
\hom(P_k,H) = \sum_{i,j} (A^k)_{ij} \leq |V(H)|^2\max_{i,j} (A^k)_{ij}.
\end{equation}
By the Perron-Frobenius theorem (see e.g. \cite[Theorem 1.5]{Seneta}), the largest eigenvalue $\lambda$ of $A$ has a strictly positive eigenvector ${\bf x}$, for which it holds that $A^k \textbf{x} = \lambda^k \textbf{x}$. Considering the row of $A^k$ containing $\max_{i,j} \left( A^k \right)_{ij}$, it follows that $\max_{i,j} \left( A^k \right)_{ij} \leq c_1 \lambda^k$ for some constant $c_1$. Combining this with (\ref{eq-homs-lin-alg-count}) we find that $\hom(P_k,H) \leq |V(H)|^2 c_1 \lambda^k$.
Since $H$ is not regular, we have $\lambda < \Delta$, and so the lemma follows.
\end{proof}

We use Lemma \ref{lem-no long paths} to show that any tree $T$ containing $P_k$ as a subgraph, for $k \geq \ell_H$, has $\hom(T,H) < |V_{=\Delta}|\Delta^{n-1}$. Indeed, by first coloring the path and then iteratively coloring the rest of the tree we obtain
\[
\hom(T,H) < \Delta^{k-2} \Delta^{n-k} < |V_{=\Delta}|\Delta^{n-1}.
\]
(The bound $\hom(P_k,H) < \Delta^{k-2}$ will be necessary for the proof of Theorem \ref{thm-2connected}; here  the weaker bound $\hom(P_k,H) < |V_{=\Delta}|\Delta^{k-1}$ would suffice.)

So suppose that $T$ does not contain a path of length $\ell_H$. 
Notice that for $n>2^{c+1}$ there are at most
\[
1+n^{1/(c+1)}+n^{2/(c+1)}+\cdots + n^{c/(c+1)} < n
\]
vertices in a rooted tree with depth at most $c$ and each vertex having degree at most $n^{1/(c+1)}$. Because of this, there exists a constant $c_2>0$ (which depends on $\ell_H$, but is independent of $n$) and 
a vertex $v \in T$ with $d(v) \geq n^{c_2}$. 
Since $T \neq K_{1,n-1}$ we have $d(v) < n-1$ and so some neighbor $w$ of $v$ is adjacent to a vertex in $V(T) \setminus (\{v\} \cup N(v))$.
As $H$ is connected and not regular, there is an $i \in V(H)$ with $d_H(i) = \Delta$ and $d_H(j) \leq \Delta-1$ for some neighbor $j$ of $i$.

The number of $H$-colorings of $T$ that don't color $v$ with a vertex of degree $\Delta$ is at most
\begin{equation} \label{eqn-cols1}
|V(H)| (\Delta-1)^{n^{c_2}}\Delta^{n-1-n^{c_2}} \leq |V(H)|e^{-n^{c_2}/\Delta}  \Delta^{n-1}.
\end{equation}
The number of $H$-colorings that color $v$ with a vertex of degree $\Delta$ different from $i$ is at most
\begin{equation} \label{eqn-cols2}
(|V_{=\Delta}| -1) \Delta^{n-1},
\end{equation}
and the number of $H$-colorings that color $v$ with $i$ and $w$ with a color different from $j$ is at most
\begin{equation} \label{eqn-cols3}
(\Delta-1)\Delta^{n-2} = \left(1-\frac{1}{\Delta}\right) \Delta^{n-1}.
\end{equation}
Finally, the number of $H$-colorings that color $v$ with $i$ and $w$ with $j$ is at most
\begin{equation} \label{eqn-cols4}
(\Delta-1)\Delta^{n-3} = \left(\frac{1}{\Delta} - \frac{1}{\Delta^2}\right) \Delta^{n-1},
\end{equation}
since the degree of $j$ is at most $\Delta-1$ and $w$ is adjacent to a vertex in $V(T) \setminus (\{v\} \cup N(v))$, which can be colored with one of the at most $\Delta-1$ neighbors of $j$.
Combining (\ref{eqn-cols1}), (\ref{eqn-cols2}), (\ref{eqn-cols3}), and (\ref{eqn-cols4}) we obtain
\[
\hom(T,H) \leq   \left( |V_{=\Delta}| - \frac{1}{\Delta^2} + |V(H)|e^{-n^{c_2}/\Delta} \right) \Delta^{n-1} < |V_{=\Delta}|\Delta^{n-1},
\]
the final inequality valid as long as $n \geq c_H$.


\section{Proof of Theorem \ref{thm-2connected}} \label{sec-2connected}

Here we build on the approach used in the proof of Theorem \ref{thm-maxfortrees} to tackle the family of $n$-vertex $2$-connected graphs $\calC_2(n)$.

In order to proceed, we will need a structural characterization of $2$-connected graphs. In the definition below we abuse standard notation a little, by allowing the endpoints of a path to perhaps coincide.
\begin{definition}
An \emph{ear} on a graph $G$ is a path whose endpoints are vertices of $G$, but which otherwise is vertex-disjoint from $G$. An ear is an \emph{open ear} 
if the endpoints of the path are distinct. 
An \emph{(open) ear decomposition} of a graph $G$ is a partition
of the edge set of $G$ into parts $Q_0, Q_1,\ldots Q_\ell$
such that $Q_0$ is a cycle and $Q_i$ for $1 \leq i \leq \ell$ is an (open) ear on $Q_0 \cup \cdots \cup Q_{i-1}$.
\end{definition}

\begin{theorem}[Whitney \cite{Whitney}]
A graph is $2$-connected if and only if it admits an open ear decomposition.
\end{theorem}

Since removing edges from a graph does not decrease the count of $H$-colorings, we will assume in the proof of Theorem \ref{thm-2connected} that we are working with a \emph{minimally $2$-connected graph} $G$, meaning that $G$ is $2$-connected and for any edge $e$ we have that $G-e$ is not $2$-connected. A characterization of these graphs is the following, which can be found in \cite{Dirac}.
\begin{theorem} \label{thm-nochord}
A $2$-connected graph $G$ is minimally $2$-connected if and only if no cycle in $G$ has a chord.
\end{theorem}

\begin{corollary}\label{cor-P3}
Let $G$ be a minimally $2$-connected graph. There is an open ear decomposition $Q_0, \ldots, Q_\ell$ of $G$, and a $c$ satisfying $1 \leq c \leq \ell$, with the property that each of $Q_1$ up to $Q_c$ are paths on at least four vertices, and each of $Q_{c+1}$ through $Q_\ell$ are paths on three vertices with endpoints in $\cup_{k=1}^c Q_k$.
\end{corollary}

\begin{proof}
Let $Q_0, \ldots, Q_\ell$ be an open ear decomposition of $G$. If any of the $Q_k$, $1 \leq k \leq \ell$, is a path on two vertices, its removal leads to an open ear decomposition of a proper spanning subgraph of $G$, contradicting the minimality of $G$. So we may assume that each of $Q_1$ through $Q_\ell$ is a path on at least three vertices. We claim that if $Q_i$ is a path on exactly three vertices, then the degree 2 vertex in $Q_i$ also has degree 2 in $G$; from this the corollary easily follows.

To verify the claim, let $Q_i$ be on vertices $x, a$ and $y$, with $a$ the vertex of degree 2, and suppose, for a contradiction, that there is some $Q_j$, $j > i$, that has endpoints $a$ and $z$ (the latter of which may be one of $x$, $y$). In $\cup_{k=1}^{i-1} Q_k$ there is a cycle $C$ containing both $x$ and $y$, and in  $\cup_{k=1}^{j-1} Q_k$ there is a path $P$ from $z$ to $C$ that intersects $C$ only once. Now consider the cycle $C'$ that starts at $a$, follows $Q_j$ to $z$, follows $P$ to $C$, follows $C$ until it has met both $x$ and $y$ (meeting $y$ second, without loss of generality), and finishes along the edge $ya$. The edge $xa$ is a chord of this cycle, giving us the desired contradiction.
\end{proof}

The next lemma follows from results that appear in \cite{Engbers}, specifically Lemma 5.3 and the proof of Corollary 5.4 of that reference.

\begin{lemma}\label{lem-col paths}
Suppose $H$ is not $K_{\Delta,\Delta}$ or $K_{\Delta}^{\circ}$ (the complete looped graph). Then for any two vertices $i$, $j$ of $H$ and for $k\geq 4$ there are at most $(\Delta^2-1)\Delta^{k-4}$ $H$-colorings of $P_k$ that map the initial vertex of the path to $i$ and the terminal vertex to $j$.
\end{lemma}

\begin{remark}
Corollary 5.4 in \cite{Engbers} gives a bound of $\Delta^{k-2}$ for a smaller class of $H$, which is simply for convenience.  The proof given actually delivers a bound of $(\Delta^2-1)\Delta^{k-4}$ for all $H$ except $K_{\Delta,\Delta}$ and $K_{\Delta}^{\circ}$. 
\end{remark}

\begin{proof}[Proof of Theorem \ref{thm-2connected}]

We begin with the proof for a non-regular connected graph $H$, and then consider regular $H$ and describe the necessary modifications needed in the proof.

Let $H$ be non-regular and connected. We will first show that for all sufficiently large $n$ and all minimally $2$-connected graphs $G$ that are different from $K_{2,n-2}$ we have $\hom(G,H) < \hom(K_{2,n-2},H)$, and we will then address the cases of equality when $G$ is allowed to be not minimal.
An easy lower bound on the number of $H$-colorings of $K_{2,n-2}$ is
$$
\hom(K_{2,n-2},H) \geq s(H)\Delta^{n-2} \geq \Delta^{n-2},
$$
where $s(H)$ is the number of ordered pairs $(i,j)$ of (not necessarily distinct) vertices in $V(H)$ with the property that $|N(i) \cap N(j)|=\Delta$. For the first inequality, consider coloring the vertices in the partition class of size $2$ of $K_{2,n-2}$ with colors $i$ and $j$, and for the second note that the pair $(i,i)$, where $i$ is any vertex of degree $\Delta$, is counted by $s(H)$.

Suppose that $G$ contains a path on $k\geq \ell_H$ vertices (with $\ell_H$ as given by Lemma \ref{lem-no long paths}). By first coloring this path and then iteratively coloring the remaining vertices we have (using Lemma \ref{lem-no long paths}) $\hom(G,H) < \Delta^{k-2} \Delta^{n-k} = \Delta^{n-2}$. We may therefore assume that $G$ contains no path on $k \geq \ell_H$ vertices.

Consider now an open ear decomposition of $G$ satisfying the conclusions of Corollary \ref{cor-P3}. Since $G$ contains no path on $\ell_H$ vertices, we have that there is some constant (independent of $n$) that bounds the lengths of each of the $Q_i$, and so $\ell$, the number of open ears in the decomposition, satisfies $\ell=\Omega(n)$.
We now show that $c$, the number of paths in the open ear decomposition that have at least four vertices, may be taken to be at most a constant (independent of $n$). Coloring $Q_0$ first, then coloring each of $Q_1$ through $Q_{c}$, and then iteratively coloring the rest of $G$, Lemma \ref{lem-col paths} yields
\begin{eqnarray*}
\hom(G,H) & \leq & |V(H)|\Delta^{|Q_0|-1}(\Delta^2-1)^{c} \Delta^{\left(\sum_{i=1}^{c} |Q_i|\right)-4c}\Delta^{\ell-c} \\
& \leq & \frac{|V(H)|}{\Delta}\left(1-\frac{1}{\Delta^2}\right)^{c} \Delta^n
\end{eqnarray*}
(noting that $n=|Q_0| + \left(\left(\sum_{i=1}^{c} |Q_i|\right)-2c\right) + (\ell-c)$). Unless $c$ is a constant, this quantity falls below the trivial lower bound on $\hom(K_{2,n-2},H)$ for all sufficiently large $n$.

Now since there are only constantly many vertices in the graph $G'$ with open ear decomposition $Q_0, \ldots, Q_{c}$, and $G$ is obtained by adding the remaining vertices to $G'$, each joined to exactly two vertices of $G'$, it follows by the pigeonhole principle that for some constant $c'$ 
there is a pair of vertices $w_1$, $w_2$ in $G$ with at least $c' n$ common neighbors, all among the middle vertices of the $Q_i$ for $i \geq {c}+1$.
(Notice that this makes $G$ ``close'' to $K_{2,n-2}$ in the same sense that in the proof of Theorem \ref{thm-maxfortrees} a tree with a vertex of degree $\Omega(n)$ was ``close'' to $K_{1,n-1}$.)

We count the number of $H$-colorings of $G$ by first considering those in which $w_1$ is colored $i$ and $w_2$ is colored $j$, for some pair $i, j \in V(H)$ with $|N(i) \cap N(j)| < \Delta$. There are at most
\begin{equation} \label{small-common-nhood}
|V(H)|^2(\Delta-1)^{c' n}\Delta^{n-2-c' n}
\end{equation}
such $H$-colorings. Next, we count the number of $H$-colorings of $G$ in which $w_1$ is colored $i$ and $w_2$ is colored $j$, for some pair $i,j \in V(H)$ with $|N(i) \cap N(j)| = \Delta$. We argue that in $G'$ (the graph with open ear decomposition $Q_0, \ldots, Q_{c}$) there must be a path $P$ on at least $4$ vertices with endpoints $w_1$ and $w_2$. 
To see this, note that since $G \neq K_{2,n-2}$ there must be some vertex $v\in G'$ so that $v \neq w_1$, $v \neq w_2$, and $v \notin N(w_1) \cap N(w_2)$. Choose a cycle in $G'$ containing $v$ and $w_1$, and find a path from $w_2$ to that cycle (the path will be trivial if $w_2$ is on the cycle). From this structure we find such a path $P$, and it must have at least $4$ vertices as $v \notin N(w_1) \cap N(w_2)$.

Coloring $G$ by first coloring $w_1$ and $w_2$, then the vertices of $P$, and finally the rest of the graph, we find by Lemma \ref{lem-col paths} that the number of
$H$-colorings of $G$ in which $w_1$ is colored $i$ and $w_2$ is colored $j$, for some pair $i,j \in V(H)$ with $|N(i) \cap N(j)| = \Delta$ is at most
\begin{equation} \label{large-common-nhood}
s(H)\left(1-\frac{1}{\Delta^2}\right)\Delta^{n-2}.
\end{equation}
Combining (\ref{small-common-nhood}) and (\ref{large-common-nhood}) we have
\begin{eqnarray*}
\hom(G,H) &\leq& |V(H)|^2\Delta^{n-2}\left(\frac{\Delta-1}{\Delta}\right)^{c' n} + s(H)\left(1-\frac{1}{\Delta^2}\right) \Delta^{n-2}\\
&<& s(H)\Delta^{n-2}\\
&\leq& \hom(K_{2,n-2},H),
\end{eqnarray*}
with the strict inequality valid for all sufficiently large $n$.

\medskip

We have shown that for non-regular connected $H$, $K_{2,n-2}$ is the unique minimally $2$-connected $n$-vertex graph maximizing the number of $H$-colorings. We now complete the proof of Theorem \ref{thm-2connected} for non-regular connected $H$ by showing that if 
$G$ is an $n$-vertex $2$-connected graph that is not minimally $2$-connected then $\hom(G,H) < \hom(K_{2,n-2},H)$.  Suppose for the sake of contradiction that $\hom(G,H) = \hom(K_{2,n-2},H)$. Since deleting edges from $G$ cannot decrease the count of $H$-colorings, the uniqueness of $K_{2,n-2}$ among minimally $2$-connected graphs shows that $K_{2,n-2}$ is the only minimally $2$-connected graph that is a subgraph of $G$, and in particular we may assume that $G$ is obtained from $K_{2,n-2}$ by adding a single edge, necessarily inside one partition class of $K_{2,n-2}$. Since $K_{2,n-2}$ is a subgraph of $G$, all $H$-colorings of $G$ are $H$-colorings of $K_{2,n-2}$.

Suppose that $i$ and $j$ are distinct adjacent vertices of $H$. The $H$-coloring of $K_{2,n-2}$ that maps one partition class to $i$ and the other to $j$ must also be an $H$-coloring of $G$, which implies that both $i$ and $j$ must be looped. By similar reasoning, if $i$ and $j$ are adjacent in $H$, and also $j$ and $k$, then $i$ and $k$ must be adjacent. It follows that $H$ is a disjoint union of fully looped complete graphs, and so if connected must be a single complete looped graph (and therefore is regular), which is a contradiction.     

\medskip

Now we turn to regular connected $H$ satisfying $s(H) \geq 2\Delta^2+1$ (if $H$ is loopless and bipartite) or $s(H) \geq \Delta^2+1$ (otherwise). For these $H$, Lemma \ref{lem-no long paths} does not apply. We will argue, however, that there is still an $\ell_H$ such that if $2$-connected $n$-vertex $G$ has an open ear decomposition in which any of the added paths $Q_i$, $i \geq 1$, has at least $\ell_H$ vertices then $G$ has fewer $H$-colorings than $K_{2,n-2}$. Once we have established this, the proof proceeds exactly as in the non-regular case.

We begin by establishing that $G$ has no long cycles; we will use the easy fact that the number of $H$-colorings of a $k$-cycle is the sum of the $k$th powers of the eigenvalues of the adjacency matrix of $H$. By the Perron-Frobenius theorem there is exactly one eigenvalue of the adjacency matrix of $H$ equal to $\Delta$, and a second one equal to $-\Delta$ if and only if $H$ is loopless and bipartite, with the remaining eigenvalues having absolute value strictly less than $\Delta$. It follows that the number of $H$-colorings of a $k$-cycle is, for large enough $k$, less than $(2+\frac{1}{\Delta^2})\Delta^k$ (if $H$ is loopless and bipartite) or $(1+\frac{1}{\Delta^2})\Delta^k$ (otherwise). If $G$ has such a cycle then by coloring the cycle first and then coloring the remaining vertices sequentially, bounding the number of options for the color of each vertex by $\Delta$, we get that the number of $H$-colorings of $G$ falls below the trivial lower bound on ${\rm hom}(K_{2,n-2},H)$ of $s(H)\Delta^{n-2}$.

Now if a $2$-connected $G$ has an open ear decomposition with an added path $Q_i$, $i \geq 1$, on at least $\ell_H$ vertices, then it has a cycle of length at least $\ell_H$, since the endpoints of $Q_i$ are joined by a path in $Q_0\cup \ldots \cup Q_{i-1}$.

\end{proof}

\section{Proof of Theorem \ref{thm-2conn-cycles}} \label{sec-2conn-cycles}

In this section we show that for all $q \geq 3$ and $n \geq 3$, among $n$-vertex $2$-edge-connected graphs, and therefore among the subfamily of $n$-vertex $2$-connected graphs, the cycle $C_n$ uniquely (up to one small exception) admits the greatest number of proper $q$-colorings (that is,  $K_q$-colorings). The result is trivial for $n=3$ and easily verified directly for $n=4$, so throughout we assume $n \geq 5$.

We will use the following characterization of $2$-edge-connected graphs, due to Robbins (see Section \ref{sec-2connected} for the definition of ear decomposition).

\begin{theorem}[Robbins \cite{Robbins}]
A graph is $2$-edge-connected if and only if it admits an ear decomposition.
\end{theorem}

Notice that we are trying to show that $C_n$, which is the unique graph constructed with a trivial ear decomposition $Q_0=C_n$, maximizes $\hom(G,K_q)$. To do this, we aim to show that almost any time a $2$-edge-connected graph contains a cycle and a path that meets the cycle only at the endpoints (which is what results after the first added ear in an ear decomposition for $G \neq C_n$), we can produce an upper bound on $\hom(G,K_q)$ which is smaller than $\hom(C_n,K_q)$; we will then deal with the exceptional cases by hand. The proof will initially consider an arbitrary fixed cycle and a path joined to that cycle at its endpoints in $G$ (but independent of a particular ear decomposition of $G$), and will begin by producing an upper bound on $\hom(G,K_q)$ based on the lengths of the path and cycle. 

Let $G \neq C_n$ be an $n$-vertex $2$-edge-connected graph, and let $C$ be a non-Hamiltonian cycle in $G$ of length $\ell$.
Since $\ell < n$ and $G$ is $2$-edge-connected, there is a path $P$ on at least $3$ vertices that meets $C$ only at its endpoints. 

Suppose first that $P$ contains $m+2$ vertices, where $m \geq 2$ (so there are $m$ vertices on $P$ outside of the vertices of $C$). Color $G$ by first coloring $C$, then $P$, and then the rest of the graph. Using  $\hom(C_{\ell},K_q) = (q-1)^{\ell} + (-1)^{\ell} (q-1)$ and Lemma \ref{lem-col paths}, we have
\begin{eqnarray*}
\hom(G,K_q) &\leq& \left((q-1)^{\ell} + (-1)^{\ell} (q-1)\right) \left[ \left( (q-1)^{2} -1\right) (q-1)^{m-2} \right] (q-1)^{n-m-\ell}\\
&=& (q-1)^n - (q-1)^{n-2} + (-1)^{\ell} (q-1)^{n-\ell+1} - (-1)^{\ell} (q-1)^{n-\ell-1}\\
&\leq& (q-1)^n - (q-1)^{n-2} + (q-1)^{n-\ell+1} - (q-1)^{n-\ell-1}\\
&\leq& \hom(C_n,K_q).
\end{eqnarray*}
Since the last inequality is strict for $\ell > 3$ and the second-to-last inequality is strict for all odd $\ell$, the chain of inequalities is strict for all $\ell \geq 3$. 
In other words, if $G$ contains any cycle $C$ and a path $P$ on at least $4$ vertices that meets $C$ only at its endpoints, then $\hom(G,K_q) < \hom(C_n,K_q)$.

Suppose next that all the paths on at least three vertices that meet $C$ only at their endpoints have exactly three vertices. If $\ell \geq 5$, or if $\ell=4$ and there is such a path that joins two adjacent vertices of the cycle, then using just the vertices and edges of $C$ and one such path $P$ we can easily find a new cycle $C'$ 
and a path $P'$ on at least four vertices that only meets $C'$ at it endpoints, and so $\hom(G,K_q) < \hom(C_n,K_q)$ as before.

It remains to find an upper bound for those $2$-edge-connected graphs (apart from $C_n$) that do not have any of the following:
\begin{itemize}
\item[(a)] a cycle of any length with a path on at least $4$ vertices that only meets the cycle at its endpoints;
\item[(b)] a cycle of length at least $5$ and a path on $3$ vertices that only meets the cycle at its endpoints; or
\item[(c)] a cycle of length $4$ with a path of length $3$ that only meets the cycle at its endpoints, and those endpoints are two adjacent vertices of the cycle.
\end{itemize}

Recall that we are assuming $n \geq 5$. For the remaining graphs, notice that if $Q_0$ is a cycle of length $3$, then $Q_0 \cup Q_1$, where $Q_1$ is a path on three vertices (by (a) above), must contain a cycle of length $4$. 
In fact, in this case this cycle will use vertices from both $Q_0$ and $Q_1$.
Therefore, we need not separately analyze the situation where $C$ is a cycle of length $3$ among the remaining graphs. 
In particular, this means that there is one final case to consider: $C$ has length $\ell=4$, and $P$ is a path on $3$ vertices that only meets $C$ at its endpoints, and the endpoints of $P$ are non-adjacent vertices of $C$.  In this case, $C$ and $P$ form a copy of $K_{2,3}$, and we color this copy of $K_{2,3}$ first and then the rest of $G$.  This gives 
$$
\hom(G,K_q) \leq \left( q(q-1)^3 + q(q-1)(q-2)^3\right) (q-1)^{n-5}.
$$
For $n \geq 6$, some algebra shows that the right-hand side above is strictly less than $\hom(C_n,K_q)$, and also for $n=5$ and $q > 3$. For $n=5$ and $q=3$ we have equality, and so in this case $\hom(G,K_3) \leq \hom(C_5,K_3)$ with equality only if $G$ has the same number of proper $3$-colorings as $K_{2,3}$ and has $K_{2,3}$ as a subgraph; this can only happen if $G=K_{2,3}$ (using a similar argument as the one given in the proof of the cases of equality for Theorem \ref{thm-2connected}).

\section{Concluding Remarks}\label{sec-concluding remarks}
In this section, we highlight a few questions related to the work in this article.

\begin{question}
Which graphs $H$ have the property that the path on $n$ vertices is the tree that minimizes the number of $H$-colorings?
\end{question}

Since the star maximizes the number of $H$-colorings among trees, it is natural to repeat the maximization question among trees with some prescribed bound on the maximum degree.  Some results (for strongly biregular $H$) are obtained in \cite{Ray}.

\begin{question}
For a fixed non-regular $H$ and positive integer $\Delta$, which tree on $n$ vertices with maximum degree $\Delta$ has the most $H$-colorings?
\end{question}

There are still open questions about regular $H$ in the family of $2$-connected graphs. In particular, we have seen that $K_{2,n-2}$ and $C_n$ are $2$-connected graphs that have the most $H$-colorings for some $H$. 

\begin{question}
What are the necessary and sufficient conditions on $H$ so that $K_{2,n-2}$ is the unique 2-connected graph that has the most $H$-colorings?
\end{question}

\begin{question}
Are $C_n$ and $K_{2,n-2}$ the only $2$-connected graphs that uniquely maximize the number of $H$-colorings for some $H$?
\end{question}

Finally, it would be natural to try to extend these results to $k$-connected graphs and to $k$-edge-connected graphs. 

\begin{question}\label{q-kconn}
Let $H$ be fixed.  Which $k$-connected graph has the most $H$-colorings? Which $k$-edge-connected graph has the most number of $H$-colorings?
\end{question}
Some results related to Question \ref{q-kconn} for graphs with fixed minimum degree $\delta$ can be found in \cite{Engbers}, with $K_{\delta,n-\delta}$ shown to be the graph with the most $H$-colorings in this family. We remark that family of connected graphs with fixed minimum degree $\delta$ has been considered very recently in \cite{Engbers1}, where it is shown that for connected non-regular $H$ and all large enough $n$, again $K_{\delta,n-\delta}$ is the unique maximizer of the count of $H$-colorings. The question for regular $H$ remains fairly open.



\begin{thebibliography}{99}

\bibitem{BrightwellWinkler}
G. Brightwell and P. Winkler, Graph homomorphisms and phase transitions, {\em J. Comb. Theory Ser. B} {\bf 77} (1999), 415-435.

\bibitem{Cutler}
J. Cutler, Coloring graphs with graphs: a survey, {\em Graph Theory Notes N.Y.} {\bf 63} (2012), 7-16.

\bibitem{CsikvariLin}
P. Csikv\'{a}ri and Z. Lin, Graph homomorphisms between trees, {\em Electronic J. Combinatorics} {\bf 21(4)} (2014), \#P4.9. 

\bibitem{Dirac}
G. A. Dirac, Minimally 2-connected graphs. {\em J. Reine Angew. Math.} {\bf  228} (1967), 204-216.

\bibitem{Engbers}
J. Engbers, Extremal $H$-colorings of graphs with fixed minimum degree, {\em J. Graph Theory} {\bf 79} (2015), 103-124.

\bibitem{Engbers1}
J. Engbers, Maximizing $H$-colorings of connected graphs with fixed minimum degree, http://arxiv.org/abs/1601.05040.

\bibitem{Hoffman}
A.J. Hoffman, Three observations on nonnegative matrices, {\em J. Res. Natl. Bur. Stand.} {\bf 71B} (1967), 39-41.

\bibitem{HuaZhang}
H. Hua and S. Zhang, Graphs with given number of cut vertices and extremal Merrifield$-$Simmons index, {\em Discrete Appl. Math.} {\bf 159} (2011), 971-980.

\bibitem{Leontovich}
A.M. Leontovich, The number of mappings of graphs, the ordering of graphs and the Muirhead theorem, {\em Problemy Peredachi Informatsii} {\bf 25} (1989), no. 2, 91-104; translation in {\em Problems Inform. Transmission} {\bf 25} (1989), 154-165.

\bibitem{London}
D. London, Two inequalities in nonnegative symmetric matrices, {\em Pacific J. Math.} {\bf 16} (1966), 515-536.

\bibitem{Lovasz}
L. Lov\'asz, Large Networks and Graph Limits, American Mathematical Society Colloquium Publications vol. 60, Providence, Rhode Island, 2012.

\bibitem{ProdingerTichy}
H. Prodinger and R. Tichy, Fibonacci numbers of graphs, {\em The Fibonacci Quart.} {\bf 20} (1982), 16-21.

\bibitem{Ray}
A. Ray, Extremal trees and reconstruction, Ph.D. Thesis, University of Nebraska-Lincoln, May 2011.

\bibitem{Robbins}
H. Robbins, A theorem on graphs with an application to a problem of traffic control, {\em Amer. Math. Monthly} {\bf 46} (1939), 281-283.

\bibitem{Seneta}
E. Seneta, Non-negative matrices and Markov chains, Springer Series in Statistics no. 21, New York, 2006.

\bibitem{Sidorenko}
A. Sidorenko, A partially ordered set of functionals corresponding to graphs, {\em Discrete Math.} {\bf 131} (1994), 263-277.

\bibitem{Sidorenko2}
A. Sidorenko, Proof of London's conjecture on sums of elements of positive matrices, {\em Math. Notes} {\bf 38} (1985), 716-717. 

\bibitem{Whitney}
H. Whitney, Non-separable and planar graphs, {\em Trans. Amer. Math. Soc.} {\bf 34} (1932), 339-362.

\bibitem{WidomRowlinson} 
B. Widom and J. S. Rowlinson, New Model for the Study of Liquid-Vapor
Phase Transitions, {\em J. Chem. Phys.} {\bf  52} (1970), 1670-1684.


\end{thebibliography}
\end{document}